\documentclass[11pt,reqno]{amsart}

\usepackage[latin1]{inputenc}
\usepackage[left=3.5cm,right=3.5cm,top=3.5cm,bottom=3.5cm]{geometry}
\usepackage{amssymb}
\usepackage{graphicx}
\usepackage{amscd}
\usepackage[hidelinks]{hyperref}
\usepackage{color}
\usepackage{float}
\usepackage{graphics,amsmath,amssymb}
\usepackage{amsthm}
\usepackage{amsfonts}
\usepackage{latexsym}
\usepackage{epsf}
\usepackage{enumerate}
\usepackage{xifthen}
\usepackage{mathrsfs}
\usepackage{dsfont}
\usepackage{makecell}
\usepackage[FIGTOPCAP]{subfigure}
\usepackage{amsmath}
\allowdisplaybreaks[4]
\usepackage{listings}
\usepackage{etoolbox}
\usepackage{fancyhdr}

 \newtheoremstyle{mytheorem}
 {3pt}
 {3pt}
 {\slshape}
 {}
 {\bfseries}
 {.}
 { }
 {}

\numberwithin{equation}{section}

\theoremstyle{mytheorem}
\newtheorem{theorem}{Theorem}[section]

\newtheorem{lemma}[theorem]{Lemma}

\newtheorem{claim}[theorem]{Claim}

\theoremstyle{definition}

\newtheorem{question}{Question}[section]

\newtheorem{remark}{Remark}[section]

\newcommand{\Keywords}[1]{\ifthenelse{\isempty{#1}}{}{\smallskip \smallskip \noindent \textbf{Keywords}. #1}}
\newcommand{\MSC}[2][2010]{\ifthenelse{\isempty{#2}}{}{\smallskip \smallskip \noindent \textbf{#1MSC}. #2}}
\newcommand{\abstractnote}[1]{\ifthenelse{\isempty{#1}}{}{\smallskip \smallskip \noindent \textsuperscript{\dag}#1}}

\makeatletter
\def\specialsection{\@startsection{section}{1}%
  \z@{\linespacing\@plus\linespacing}{.5\linespacing}%
  {\normalfont}}
\def\section{\@startsection{section}{1}%
  \z@{.7\linespacing\@plus\linespacing}{.5\linespacing}%
  {\normalfont\scshape}}
\patchcmd{\@settitle}{\uppercasenonmath\@title}{\Large}{}{}
\patchcmd{\@settitle}{\begin{center}}{\begin{flushleft}}{}{}
\patchcmd{\@settitle}{\end{center}}{\end{flushleft}}{}{}
\patchcmd{\@setauthors}{\MakeUppercase}{\normalsize}{}{}
\patchcmd{\@setauthors}{\centering}{\raggedright}{}{}
\patchcmd{\section}{\scshape}{\large\bfseries}{}{}
\renewcommand{\@secnumfont}{\bfseries}
\patchcmd{\@startsection}{\@afterindenttrue}{\@afterindentfalse}{}{}
\patchcmd{\abstract}{\leftmargin3pc}{\leftmargin1pc}{}{}

\def\maketitle{\par
  \@topnum\z@ 
  \@setcopyright
  \thispagestyle{empty}
  \ifx\@empty\shortauthors \let\shortauthors\shorttitle
  \else \andify\shortauthors
  \fi
  \@maketitle@hook
  \begingroup
  \@maketitle
  \toks@\@xp{\shortauthors}\@temptokena\@xp{\shorttitle}%
  \toks4{\def\\{ \ignorespaces}}
  \edef\@tempa{%
    \@nx\markboth{\the\toks4
      \@nx\MakeUppercase{\the\toks@}}{\the\@temptokena}}%
  \@tempa
  \endgroup
  \c@footnote\z@
  \@cleartopmattertags
}
\makeatother

\title{Remarks on the distribution of the primitive roots of a prime}

\author[S. Chern]{Shane Chern}
\address{School of Mathematical Sciences, Zhejiang University, Hangzhou, 310027, China}
\email{\href{mailto:shanechern@zju.edu.cn}{shanechern@zju.edu.cn}; \href{mailto:chenxiaohang92@gmail.com}{chenxiaohang92@gmail.com}}

\date{}

\begin{document}

%

\maketitle

\begin{abstract}
Let $\mathbb{F}_p$ be a finite field of size $p$ where $p$ is an odd prime. Let $f(x)\in\mathbb{F}_p[x]$ be a polynomial of positive degree $k$ that is not a $d$-th power in $\mathbb{F}_p[x]$ for all $d\mid p-1$. Furthermore, we require that $f(x)$ and $x$ are coprime. The main purpose of this paper is to give an estimate of the number of pairs $(\xi,\xi^\alpha f(\xi))$ such that both $\xi$ and $\xi^\alpha f(\xi)$ are primitive roots of $p$ where $\alpha$ is a given integer. This answers a question of Han and Zhang.


\MSC{Primary 11A07; Secondary 11L40.}
\end{abstract}

\section{Introduction}

Let $a$ and $q$ be relatively prime integers, with $q\ge 1$. We know from the Euler-Fermat theorem that $a^\phi(q)\equiv 1\bmod q$, where $\phi(q)$ is the Euler totient function. We say an integer $f$ is the exponent of $a$ modulo $q$ if $f$ is smallest positive integer such that $a^f\equiv 1\bmod q$. If $f=\phi(q)$, then $a$ is called a primitive root of $q$. If $q$ has a primitive root $a$, then the group of the reduced residue classes mod $q$ is the cyclic group generated by the residue class $\hat{a}$. It is well-known that primitive roots exist only for the following moduli:
$$q=1,\ 2,\ 4,\ p^\alpha,\ \text{and}\ 2p^\alpha,$$
where $p$ is an odd prime and $\alpha\ge 1$. The reader may refer to Chapter 10 of T. M. Apostol's book \cite{Apo1976} for detailed contents.

There has been a long history studying the distribution of the primitive roots of a prime. In a recent paper, D. Han and W. Zhang \cite{Han2015} considered the number of pairs $(\xi,m\xi^k+n\xi)$ such that both $\xi$ and $m\xi^k+n\xi$ are primitive roots of an odd prime $p$ where $m$, $n$ and $k$ are given integers with $k\ne 1$ and $(mn,p)=1$. The reader may also find some descriptions of other interesting problems on primitive roots such as the Golomb's conjecture in \cite{Han2015} and references therein. After presenting their main results, Han and Zhang proposed the following
\begin{question}
Let $\mathbb{F}_p$ be a finite field of size $p$ and $f(x)$ be an irreducible polynomial in $\mathbb{F}_p[x]$. Whether there exists a primitive element $\xi\in\mathbb{F}_p$ such that $f(\xi)$ is also a primitive element in $\mathbb{F}_p$?
\end{question}

In this paper, we let $f(x)\in\mathbb{F}_p[x]$ be a polynomial of positive degree $k$ that is not a $d$-th power in $\mathbb{F}_p[x]$ for all $d\mid p-1$. Furthermore, we require that $f(x)$ and $x$ are coprime. Let $\alpha$ be a given integer, we denote by $N(\alpha,f;p)$ the number of pairs $(\xi,\xi^\alpha f(\xi))$ such that both $\xi$ and $\xi^\alpha f(\xi)$ are primitive roots of $p$. Our result is

\begin{theorem}
It holds that
\begin{equation}
N(\alpha,f;p)=(p-1-R(f))\left(\frac{\phi(p-1)}{p-1}\right)^2+ \theta k 4^{\omega(p-1)}\sqrt{p}\left(\frac{\phi(p-1)}{p-1}\right)^2,
\end{equation}
where $|\theta|< 1$, $\omega(n)$ denotes the number of distinct prime divisors of $n$, and $R(f)$ denotes the number of distinct zeros of $f(x)$ in $\mathbb{F}_p$.
\end{theorem}

Now if we take $\alpha=0$ and $f(x)=x+1$, then we get the famous result on consecutive primitive roots obtained by J. Johnsen \cite{Joh1971} and M. Szalay \cite{Saz1975}. If we take
\begin{equation*}
\begin{cases}
\alpha=1 \text{ and }f(x)=mx^{k-1}+n & \text{if }k>1,\\
\alpha=k \text{ and }f(x)=nx^{1-k}+m & \text{if }k<1,
\end{cases}
\end{equation*}
then we get Han and Zhang's result immediately.

\begin{remark}
We should mention that there is a minor mistake in Han and Zhang's result. In fact, they forgot to consider the zeros of $f(x)$ in $\mathbb{F}_p$. For example, if we choose $f(x)=x^{-1}+x=x^{-1}(x^2+1)$, then there are $1+(-1|p)$ distinct zeros of $x^2+1$ in $\mathbb{F}_p$ where $(*|p)$ is the Legendre symbol. In this sense, the main term of $N(-1,x^2+1;p)$ (or their $N(-1,1,1,p)$) should be
$$(p-2-(-1|p))\left(\frac{\phi(p-1)}{p-1}\right)^2$$
while not $\phi^2(p-1)/(p-1)$.
\end{remark}

\section{Preliminary lemmas}

We first introduce the indicator function of primitive roots.

\begin{lemma}[L. Carlitz {\cite[Lemma 2]{Car1956}}]\label{le:1}
We have
\begin{equation}\label{eq:le.1}
\frac{\phi(p-1)}{p-1}\sum_{d\mid p-1}\frac{\mu(d)}{\phi(d)}\sum_{\chi\bmod p\atop \mathrm{ord}\chi=d}\chi(n)=\begin{cases}
1 & \text{if $n$ is a primitive root of $p$,}\\
0 & \text{otherwise.}
\end{cases}
\end{equation}
Here $\mu$ is the M\"obius function, and $\mathrm{ord}\chi$ denotes the order of a Dirichlet character $\chi\bmod p$, that is, the smallest positive integer $f$ such that $\chi^f=\chi_0$, the principal character modulo $p$.
\end{lemma}

The following famous Weil bound for character sums plays an important role in our proof.
\begin{lemma}[A. Weil \cite{Wei1948}]\label{le:2}
Let $\chi$ be a non-principal Dirichlet character modulo $p$ with order $d$. Suppose $f(x)\in\mathbb{F}_p[x]$ is a polynomial of positive degree $k$ that is not a $d$-th power in $\mathbb{F}_p[x]$. Then we have
\begin{equation}\label{eq:le.2}
\left|\sum_{n=1}^{p-1}\chi(f(n))\right|\le (k-1)\sqrt{p}.
\end{equation}
\end{lemma}

We also need the less-known extension of Weil bound obtained by D. Wan.
\begin{lemma}[D. Wan {\cite[Corollary 2.3]{Wan1997}}]\label{le:3}
Let $\chi_1,\chi_2,\ldots,\chi_m$ be non-principal Dirichlet characters modulo $p$ with orders $d_1,d_2,\ldots,d_m$, respectively. Suppose $f_1(x),f_2(x),\ldots,$ $f_m(x)\in\mathbb{F}_p[x]$ are pairwise prime polynomials of positive degrees $k_1,k_2,\ldots,k_m$. Suppose also that $f_i(x)$ is not a $d_i$-th power in $\mathbb{F}_p[x]$ for all $i=1,2,\ldots,m$. Then we have
\begin{equation}\label{eq:le.3}
\left|\sum_{n=1}^{p-1}\chi_1(f_1(n))\chi_2(f_2(n))\cdots\chi_m(f_m(n))\right|\le \left(\sum_{i=1}^m k_i-1\right)\sqrt{p}.
\end{equation}
\end{lemma}

From Lemmas \ref{le:2} and \ref{le:3}, we have

\begin{lemma}\label{le:4}
Let $\chi_1$ be a Dirichlet character modulo $p$, and $\chi_2$ be a non-principal Dirichlet character modulo $p$ with order $d$. Suppose $f(x)\in\mathbb{F}_p[x]$ is a polynomial of positive degree $k$ that is not a $d$-th power in $\mathbb{F}_p[x]$. We also require that $f(x)$ and $x$ are coprime. Furthermore, let $\alpha$ be a given integer. Then we have
\begin{equation}\label{eq:le.42}
\left|\sum_{n=1}^{p-1}\chi_1(n^\alpha)\chi_2(f(n))\right|\le \begin{cases}
(k-1)\sqrt{p} & \text{if $\chi_1^\alpha$ is the principal character,}\\
k\sqrt{p} & \text{otherwise.}
\end{cases}
\end{equation}
\end{lemma}

\begin{proof}
Note that
$$\sum_{n=1}^{p-1}\chi_1(n^\alpha)\chi_2(f(n))=\sum_{n=1}^{p-1}\chi_1^\alpha(n)\chi_2(f(n)).$$
Now if $\chi_1^\alpha$ is the principal character, then it follows that
$$\sum_{n=1}^{p-1}\chi_1(n^\alpha)\chi_2(f(n))=\sum_{n=1}^{p-1}\chi_2(f(n)),$$
and we get the bound from Lemma \ref{le:2}. If $\chi_1^\alpha$ is not the principal character, then the bound is obtained through a direct application of Lemma \ref{le:3}.
\end{proof}

\section{Proof of the main result}
It follows by Lemma \ref{le:1} that
\begin{align*}
&N(\alpha,f;p)\\
&\quad =\sum_{n=1}^{p-1}\left(\frac{\phi(p-1)}{p-1}\right)^2\sum_{d_1\mid p-1}\sum_{d_2\mid p-1}\frac{\mu(d_1)}{\phi(d_1)}\frac{\mu(d_2)}{\phi(d_2)}\sum_{\chi_1\bmod p\atop \mathrm{ord}\chi_1=d_1}\sum_{\chi_2\bmod p\atop \mathrm{ord}\chi_2=d_2}\chi_1(n)\chi_2(n^\alpha f(n))\\
&\quad=(p-1-R(f))\left(\frac{\phi(p-1)}{p-1}\right)^2\\
&\quad\quad\quad+\left(\frac{\phi(p-1)}{p-1}\right)^2\sum_{d_1\mid p-1\atop d_1>1}\frac{\mu(d_1)}{\phi(d_1)}\sum_{\chi_1\bmod p\atop \mathrm{ord}\chi_1=d_1}\sum_{n-1}^{p-1}\chi_1(n)\\
&\quad\quad\quad+\left(\frac{\phi(p-1)}{p-1}\right)^2\sum_{d_2\mid p-1\atop d_2>1}\frac{\mu(d_2)}{\phi(d_2)}\sum_{\chi_2\bmod p\atop \mathrm{ord}\chi_2=d_2}\sum_{n-1}^{p-1}\chi_2(n^\alpha f(n))\\
&\quad\quad\quad+\left(\frac{\phi(p-1)}{p-1}\right)^2\sum_{d_1\mid p-1\atop d_1>1}\sum_{d_2\mid p-1\atop d_2>1}\frac{\mu(d_1)}{\phi(d_1)}\frac{\mu(d_2)}{\phi(d_2)}\sum_{\chi_1\bmod p\atop \mathrm{ord}\chi_1=d_1}\sum_{\chi_2\bmod p\atop \mathrm{ord}\chi_2=d_2}\sum_{n=1}^{p-1}\chi_1(n)\chi_2(n^\alpha f(n)).
\end{align*}

\begin{claim}\label{cl:1}
We have
$$\sum_{d_1\mid p-1\atop d_1>1}\frac{\mu(d_1)}{\phi(d_1)}\sum_{\chi_1\bmod p\atop \mathrm{ord}\chi_1=d_1}\sum_{n-1}^{p-1}\chi_1(n)=0.$$
\end{claim}

\begin{proof}
We deduce it directly from
$$\sum_{n=1}^{p-1}\chi(n)=0$$
if $\chi$ is not the principal character modulo $p$.
\end{proof}

\begin{claim}\label{cl:2}
We have
$$\left|\sum_{d_2\mid p-1\atop d_2>1}\frac{\mu(d_2)}{\phi(d_2)}\sum_{\chi_2\bmod p\atop \mathrm{ord}\chi_2=d_2}\sum_{n-1}^{p-1}\chi_2(n^\alpha f(n))\right|\le (2^{\omega(p-1)}-1)k\sqrt{p}.$$
\end{claim}

\begin{proof}
Note that
$$\sum_{n-1}^{p-1}\chi_2(n^\alpha f(n))=\sum_{n=1}^{p-1}\chi_2(n^\alpha)\chi_2(f(n)).$$
Now by Lemma \ref{le:4}, we have
$$\left|\sum_{n-1}^{p-1}\chi_2(n^\alpha f(n))\right|\le k\sqrt{p}.$$
Note also that
$$\sum_{d\mid p-1\atop d>1}\left|\mu(d)\right|=2^{\omega(p-1)}-1.$$
We therefore have
\begin{align*}
\left|\sum_{d_2\mid p-1\atop d_2>1}\frac{\mu(d_2)}{\phi(d_2)}\sum_{\chi_2\bmod p\atop \mathrm{ord}\chi_2=d_2}\sum_{n-1}^{p-1}\chi_2(n^\alpha f(n))\right|&\le \sum_{d_2\mid p-1\atop d_2>1}\left|\frac{\mu(d_2)}{\phi(d_2)}\right|\sum_{\chi_2\bmod p\atop \mathrm{ord}\chi_2=d_2}\left|\sum_{n-1}^{p-1}\chi_2(n^\alpha f(n))\right|\\
&\le \sum_{d_2\mid p-1\atop d_2>1}\left|\frac{\mu(d_2)}{\phi(d_2)}\right|\phi(d_2)k\sqrt{p}\\
&= (2^{\omega(p-1)}-1)k\sqrt{p}.
\end{align*}
\end{proof}

\begin{claim}\label{cl:3}
We have
$$\left|\sum_{d_1\mid p-1\atop d_1>1}\sum_{d_2\mid p-1\atop d_2>1}\frac{\mu(d_1)}{\phi(d_1)}\frac{\mu(d_2)}{\phi(d_2)}\sum_{\chi_1\bmod p\atop \mathrm{ord}\chi_1=d_1}\sum_{\chi_2\bmod p\atop \mathrm{ord}\chi_2=d_2}\sum_{n=1}^{p-1}\chi_1(n)\chi_2(n^\alpha f(n))\right|\le (2^{\omega(p-1)}-1)^2 k\sqrt{p}.$$
\end{claim}

\begin{proof}
Note that
$$\sum_{n=1}^{p-1}\chi_1(n)\chi_2(n^\alpha f(n))=\sum_{n=1}^{p-1}\chi_1\chi_2^\alpha(n)\chi_2(f(n)).$$
Again by Lemma \ref{le:4}, we get
$$\left|\sum_{n=1}^{p-1}\chi_1(n)\chi_2(n^\alpha f(n))\right|\le k\sqrt{p}.$$
We therefore have
\begin{align*}
&\left|\sum_{d_1\mid p-1\atop d_1>1}\sum_{d_2\mid p-1\atop d_2>1}\frac{\mu(d_1)}{\phi(d_1)}\frac{\mu(d_2)}{\phi(d_2)}\sum_{\chi_1\bmod p\atop \mathrm{ord}\chi_1=d_1}\sum_{\chi_2\bmod p\atop \mathrm{ord}\chi_2=d_2}\sum_{n=1}^{p-1}\chi_1(n)\chi_2(n^\alpha f(n))\right|\\
&\quad\le \sum_{d_1\mid p-1\atop d_1>1}\sum_{d_2\mid p-1\atop d_2>1}\left|\frac{\mu(d_1)}{\phi(d_1)}\right|\left|\frac{\mu(d_2)}{\phi(d_2)}\right|\sum_{\chi_1\bmod p\atop \mathrm{ord}\chi_1=d_1}\sum_{\chi_2\bmod p\atop \mathrm{ord}\chi_2=d_2}\left|\sum_{n=1}^{p-1}\chi_1(n)\chi_2(n^\alpha f(n))\right|\\
&\quad\le \sum_{d_1\mid p-1\atop d_1>1}\sum_{d_2\mid p-1\atop d_2>1}\left|\frac{\mu(d_1)}{\phi(d_1)}\right|\left|\frac{\mu(d_2)}{\phi(d_2)}\right|\phi(d_1)\phi(d_2)k\sqrt{p}\\
&\quad= (2^{\omega(p-1)}-1)^2 k\sqrt{p}.
\end{align*}
\end{proof}

We conclude by combining Claims \ref{cl:1}-\ref{cl:3} that
\begin{align*}
&\left|N(\alpha,f;p)-(p-1-R(f))\left(\frac{\phi(p-1)}{p-1}\right)^2\right|\\
&\quad\le \left((2^{\omega(p-1)}-1)+(2^{\omega(p-1)}-1)^2\right) k\sqrt{p}\left(\frac{\phi(p-1)}{p-1}\right)^2\\
&\quad< k 4^{\omega(p-1)}\sqrt{p}\left(\frac{\phi(p-1)}{p-1}\right)^2.
\end{align*}
This completes our proof.

\bibliographystyle{amsplain}

\end{document}